\documentclass[11pt,reqno]{article}

\usepackage{booktabs} 
\usepackage{array} 
\usepackage{paralist} 
\usepackage{verbatim} 
\usepackage{amsthm}
\usepackage[table,xcdraw]{xcolor}
\usepackage{latexsym, amsmath, amssymb, a4, epsfig, color}
\usepackage{blindtext}
\usepackage{graphicx}
\usepackage{subcaption}
\usepackage[utf8]{inputenc}
\usepackage[export]{adjustbox}
\usepackage{wrapfig}
\usepackage{apptools}

\usepackage{floatrow}

\usepackage{amssymb}
\usepackage{amsmath}
\usepackage[normalem]{ulem} 

\AtAppendix{\counterwithin{definition}{subsection}}
\AtAppendix{\counterwithin{theorem}{subsection}}

\graphicspath{}

\newtheorem{theorem}{Theorem}[section]

\newtheorem{lemma}[theorem]{Lemma}
\newtheorem{prop}[theorem]{Proposition}

\newtheorem{remark}{Remark}

\newcommand{\CC}{\mathbb{C}}
\newcommand{\NN}{\mathbb{N}}

\newcommand{\Om}{\Omega}
\newcommand{\ds}{\displaystyle}

\newcommand{\p}{\partial}

\newcommand{\eqnref}[1]{(\ref {#1})}
\renewcommand{\qed}{\hfill $\Box$ \medskip}
\newcommand{\beq}{\begin{equation}}
\newcommand{\eeq}{\end{equation}}

\newcommand{\SingleOmega}{\mathcal{S}_{\partial\Omega}}

\newcommand{\KstarOmega}{\mathcal{K}_{\partial\Omega}^{*}}

\newcommand{\Kcal}{\mathcal{K}}

\newcommand{\LtwobdOmega}{L^2(\partial\Omega)}
\newcommand{\LtwozerobdOmega}{L_0^2(\partial\Omega)}

\numberwithin{equation}{section}
\numberwithin{figure}{section}

\begin{document}

\title{A decay estimate for the eigenvalues of the Neumann-Poincar\'{e} operator using the Grunsky coefficients
 \thanks{
 \footnotesize
 This work is supported by the Korean Ministry of Science, ICT and Future Planning through NRF grant NRF-2016R1A2B4014530 (to M.L and Y.J).}
}
\author{
YoungHoon Jung\thanks{\footnotesize Department of Mathematical Sciences, Korea Advanced Institute of Science and Technology, Daejeon 305-701, Korea (hapy1010@kaist.ac.kr.)}  \and
Mikyoung Lim\thanks{\footnotesize Department of Mathematical Sciences, Korea Advanced Institute of Science and Technology, Daejeon
305-701, Korea (mklim@kaist.ac.kr).} }

\date{\today}
\maketitle
\begin{abstract}
\noindent We investigate the decay property of the eigenvalues of the Neumann-Poincar\'{e} operator in two dimensions. As is well-known, this operator admits only a sequence of eigenvalues that accumulates to zero as its spectrum for a bounded domain having $C^{1,\alpha}$ boundary with $\alpha\in (0,1)$. In this paper, we  show that the eigenvalue $\lambda_k$'s of the Neumann-Poincar\'{e} operator ordered by size satisfy that $|\lambda_k| = O(k^{-p-\alpha+1/2})$ for an arbitrary simply connected domain having $C^{1+p,\alpha}$ boundary with $p\geq 0,~ \alpha\in(0,1)$ and $p+\alpha>\frac{1}{2}$.

\end{abstract}



\section{Introduction}
In this paper, we investigate the decay property of the eigenvalues of the Neumann-Poincar\'{e} (NP) operator based on the series expansion of the NP operator that recently appeared in \cite{jung2018new}. The NP operator is a boundary integral operator that naturally shows up when one solves a transmission problem in electrostatics via the boundary integral formulation. Recently, the spectral analysis of the NP operator has drawn much attention because of its applications to nanophotonics and metamaterials \cite{ammari2013spectral,ando2016plasmon,bonnetier2012pointwise,bonnetier2013spectrum,grieser2014plasmonic,mayergoyz2005electrostatic,milton2006cloaking}. For examples, it was shown that plasmon resonance takes place at the eigenvalues of the NP operator \cite{grieser2014plasmonic} and that cloaking by anomalous localized resonance occurs at the accumulation point of the eigenvalues \cite{ando2016plasmon}.

For a simply connected bounded Lipschitz domain $\Omega\subset\mathbb{R}^2$ and a density function $\varphi\in L^2(\partial\Omega),$ the NP operator is defined by
\begin{equation}\label{eqn:Kstar}
	\KstarOmega[\varphi](x)=p.v.\frac{1}{2\pi}\int_{\partial\Omega}\frac{\left<x-y,\nu_x\right>}{|x-y|^2}\varphi(y)\,d\sigma(y).
\end{equation}
Here, $p.v$ denotes the Cauchy principal value and $\nu_x$ is the outward unit normal vector at $x\in\partial\Omega$. 
It is worth remarking that one can generalize the NP operator to be defined for a multiply connected domain \cite{ammari2004reconstruction}.

While $\KstarOmega$ is symmetric on $\LtwobdOmega$ only for a disk or a ball \cite{lim2001symmetry}, it can be realized as a symmetric operator on $H_0^{-1/2}(\partial\Omega)$ by defining a new inner product based on Plemelj's symmetrization principle \cite{khavinson2007poincare}. The space $H_0^{-1/2}(\partial\Omega)$ denotes the Sobolev space $H^{-1/2}(\partial\Omega)$ with the mean-zero condition. 
Hence, the spectrum of the NP operator on $\LtwozerobdOmega$ or $H_{0}^{-1/2}(\partial\Omega)$ lies on the real axis in the complex plane. More specifically, it is contained in $(-1/2,1/2)$ \cite{kellogg2012foundations,verchota1984layer} (see also \cite{kang2018spectral,krein1998compact}).
If $\Om$ has $C^{1,\alpha}$ boundary with some $\alpha\in(0,1)$, then $\KstarOmega$ is compact and, hence, it admits only a sequence of eigenvalues that accumulates to zero as its spectrum. For simple shapes such as disks or ellipses, the complete sets of eigenvalues are known; see for example \cite{ammari2007boundary}. If $\Om$ is merely a Lipschitz domain, then the corresponding $\KstarOmega$ admits the continuous spectrum as well as the eigenvalues. Various studies are underway to investigate the spectral properties of the NP operator for cornered domains; see for examples \cite{helsing2017classification,kang2017spectral,li2018embedded,perfekt2014spectral,perfekt2017essential}.
 
As already mentioned, the purpose of this paper is to find the decay rate of the eigenvalues of the NP operator $\Kcal^*_{\p\Om}$. Before stating our main theorem, we wish to mention two related results. As the set of eigenvalues of the NP operator in two dimensions is symmetric with respect to $0$ (see \cite{blumenfeld1914poincare}), we can let $\lambda_k$'s denote the eigenvalues ordered as in \eqnref{eqn:eigenvalueorder}. 
In \cite{ando2016exponential}, it was shown that
$$|\lambda_k|=O(e^{-k\epsilon}) \quad\mbox{for } \p\Om\mbox{ analytic}$$
with a constant $\epsilon>0$ depending on $\Om$. 
It also holds (see \cite{miyanishi2017eigenvalues}) that 
 \beq\label{rate:miyanish}
 |\lambda_k|=o(k^\beta)\quad \mbox{for }\p\Om \mbox{ of class } C^p,\ p\geq2,
 \eeq
with any $\beta>-p+3/2.$

The following is the main theorem of this paper. One can find the proof in section \ref{sec:proof}.
\begin{theorem}\label{theorem:main}
	Let $\Omega$ be a simply connected bounded domain having $C^{1+p,\alpha}$ boundary with $p\geq 0$, $\alpha\in(0,1)$ and $p+\alpha>1/2.$ Let $\{\lambda_k\}_{k=1}^\infty$ be the eigenvalus of the NP operator $\KstarOmega$ on $H^{-1/2}_0(\partial\Omega)$ enumerated in the following way:
	\begin{equation}\label{eqn:eigenvalueorder}
		0.5>\left|\lambda_1\right|=\left|\lambda_2\right|\geq\left|\lambda_3\right|=\left|\lambda_4\right|\geq\cdots.
	\end{equation}
	Then there exists a constant $C>0$ independent of $k$ such that 
	\begin{equation*}
		\left|\lambda_{2k-1}\right|=\left|\lambda_{2k}\right| \leq C k^{-p-\alpha+1/2}\quad\mbox{for all } k=1,2,3,\dots.
	\end{equation*}
\end{theorem}
Our result generalizes the decay estimates obtained in \cite{ando2016exponential,miyanishi2017eigenvalues} up to the boundary regularity $C^{1,\alpha}$ with $\alpha>1/2$ (by setting $p=0$). Furthermore, it improves the decay rate obtained in \cite{miyanishi2017eigenvalues} for $\p\Om$ of class $C^{1+p,\alpha}$ with $\alpha>0$; while the decay rate from \eqnref{rate:miyanish} is $o(k^{\beta})$ with $\beta>-p+1/2$, our result in Theorem \ref{theorem:main} gives $O(k^{-p-\alpha+1/2})$. 

The rest of the paper is organized as follows. In section \ref{sec:pre} we review some results in geometric function theory. Section \ref{sec:series} is then devoted to review the series expansion of the NP operator that recently appeared in \cite{jung2018new}. In section \ref{sec:proof} we prove the main theorem.

\section{Preliminary}\label{sec:pre}
\subsection{The Faber polynomials and the Grunsky coefficients}
We review some results in geometric function theory. For more information, we recommend that the reader see \cite{duren2001univalent,henrici1993appliedvol3,smirnov1968functions}.

Let $\Omega$ be a simply connected bounded domain in $\mathbb{R}^2$.
According to the Riemann mapping theorem there is a unique conformal mapping $\Psi$ from $\{w\in\mathbb{C}~:~|w|>\gamma\}$ onto the exterior region $\mathbb{C}\setminus\overline{\Omega}$ whose Laurent series expansion is given by
\begin{equation}\label{eqn:extmapping}
	\Psi(w)=w+{a}_0+\frac{{a}_1}{w}+\frac{{a}_2}{w^2}+\cdots,\quad |w|>\gamma.
\end{equation}
The quantity $\gamma$ is called the logarithmic capacity of $\overline{\Omega}$. We set $$\rho_0=\ln \gamma.$$

Extension properties of the conformal mapping depending on the regularity of $\p\Om$ have been studied by various authors.
For a simply connected Jordan domain, the conformal mapping $\Psi$ can be extended continuously to $\partial\Omega$ by the Carath\'eodory extension theorem \cite{duren2001univalent, henrici1974appliedvol1}. If the domain has no cusp on its boundary, then $|\Psi'(w)|$ is integrable on $\partial\Omega$ even when $\p\Om$ has a corner point. If $\p\Om$ is $C^{m,\alpha}$ for some $m\in\NN$ and $\alpha\in(0,1)$, then the $m$-th derivative $\Psi^{(m)}$ admits an extension to $\p\Om$ which is of Lipschitz continuity of order $\alpha$ on $\overline{\CC\setminus\Om}$ by the Kellogg-Warschawski theorem.
 One can find more details on the regularity results for the conformal mapping in \cite{pommerenke2013boundary}.

The mapping $\Psi$ determines (uniquely for the domain $\Om$) the Faber polynomials $\{F_m(z)\}_{m=0}^\infty$ that are monic satisfying
\begin{equation}\label{eqn:Faberdefinition}
	F_m(\Psi(w))-w^m
	=\sum_{k=1}^{\infty}c_{m,k}{w^{-k}},\quad m=1,2,\dots.
\end{equation}
The coefficients $c_{m,k}$ are called the Grunsky coefficients.

The Faber polynomials can be determined via the recursion relation
\begin{equation}\label{eqn:Faberrecursion}
	-na_n=F_{n+1}(z)+\sum_{s=0}^{n}a_{s}F_{n-s}(z)-zF_n(z),\quad n\geq 0,
\end{equation}
with the initial condition $F_0(z)=1$. Similarly, the Grunsky coefficients can be determined by
\begin{equation}\label{eqn:cnkrecursion}
	c_{m,k+1}=c_{m+1,k}-a_{m+k}+\sum_{s=1}^{m-1}a_{m-s}c_{s,k}-\sum_{s=1}^{k-1}a_{k-s}c_{m,s},\quad m,k\geq 1,
\end{equation}
with the initial condition $c_{n,1}=na_n$ for all $n\geq1$. Here, we set $\sum_{s=1}^{0}=0$.
The following Grusnky identity is well-known:
\begin{equation}\label{eqn:Grunskyidentity}
	mc_{n,m}=nc_{m,n}\quad m,n=1,2,\cdots.
\end{equation}
Using \eqnref{eqn:Grunskyidentity}, one can symmetrizie the Grunsky coefficients as
\begin{equation}\label{eqn:symGrunskycoeff}
	\mu_{m,k} = \sqrt{\frac{k}{m}} {c_{m,k}}	
\end{equation}
so that $\mu_{m,k}=\mu_{k,m}$.
Using the polynomial area theorem \cite{duren2001univalent}, it can be easily shown that the coefficients \eqnref{eqn:symGrunskycoeff} satisfy the $l^2$-type bound
\begin{equation}\label{muineq1_2}
\sum_{k=1}^\infty\left|\frac{\mu_{m,k}}{\gamma^{m+k}}\right|^2\leq 1,
\end{equation}
where strict inequality holds unless $\Omega$ has measure zero.

\subsection{Orthogonal coordinates in $\mathbb{C}\setminus\Omega$ and function space on $\partial\Omega$}

The exterior conformal mapping $\Psi$ naturally induces an orthogonal coordinate system in the exterior region $\mathbb{C}\setminus\Omega$. We associate each $z\in \mathbb{C}\setminus\Omega$ with the modified polar coordinate $(\rho,\theta)\in[\rho_0,\infty)\times[0,2\pi)$ via the relation
$$z=\Psi(e^{\rho+i\theta}).$$
For notational convenience, we let $\Psi(\rho,\theta)$ indicate $\Psi(e^{\rho+i\theta})$. It is easy to see that the scale factors $h_\rho=|\frac{\partial \Psi}{\partial\rho}|$ and $h_\theta=|\frac{\partial \Psi}{\partial\theta}|$ coincide. We denote $$h(\rho,\theta)=h_\rho=h_\theta.$$

With this curvilinear orthogonal coordinate system, the Laplacian operator for a function $u$ becomes
\begin{equation}
	\Delta u=\frac{1}{h^2(\rho,\theta)}\left(\frac{\partial^2 u}{\partial \rho^2}+\frac{\partial^2 u}{\partial \theta^2}\right)
\end{equation}
and the exterior normal derivative on $\partial\Omega$ is
 \begin{equation}
 	\frac{\partial u}{\partial \nu}\Big|_{\partial\Omega}^{+}(z)=\frac{1}{h(\rho,\theta)}\frac{\partial }{\partial \rho}u(\Psi(e^{\rho+i\theta}))\Big|_{\rho\rightarrow\rho_0^+}.\label{eqn:normalderiv}
 \end{equation}
 Here, $\nu$ is the outward unit normal vector on $\p\Om$ and the symbol $+$ (respectively, $-$) indicates the limit from the exterior (respectively, interior) of $\Om$. 
 Thanks to the property $h_\rho=h_\theta$, we have
	\begin{equation}\label{eqn:boundaryintegral}
		\int_{\partial\Omega}\frac{\partial u}{\partial\nu}\Big|^+_{\partial\Omega}(z)\, d\sigma(z)=\int_{0}^{2\pi}\frac{\partial }{\partial \rho}u(\Psi(e^{\rho+i\theta}))\Big|_{\rho\rightarrow\rho_0^+}\, d\theta.
	\end{equation}
 We remark that $h(\rho,\theta)$ and $\frac{1}{h(\rho,\theta)}$ are integrable on $\partial\Omega$ even when $\partial\Omega$ has a corner; see \cite{jung2018new} for further details.
From now on, we write $f(\theta)=(f\circ\Psi)(\rho_0,\theta)$ for a function $f$ defined on $\partial\Omega$ for the sake of simplicity.

\section{Series expansion of the NP operator in the orthogonal curvilinear coordinates}\label{sec:series}
We briefly review the results of a recent paper \cite{jung2018new}. We let $\Om$ satisfy the same assumption in Proposition \ref{thm:series}.

\subsection{Density basis functions and the Hilbert space $K^{-1/2}$ on $\p\Om$}
In terms of the orthogonal curvilinear system $(\rho,\theta)$ associated with $\Psi$, we set a system of density functions on $\p\Om$ as
\begin{equation}
	\begin{cases}
	\ds\zeta_0(z)=\frac{1}{h(\rho_0,\theta)},\\[2mm]
	\ds\zeta_m (z)= |m|^{\frac{1}{2}}\;\frac{e^{im\theta}}{h(\rho_0,\theta)},\quad m=1,2,\dots.
\end{cases}
\end{equation}
From the regularity of the conformal mapping, one can see that $\zeta_m(z)\in L^2(\partial\Omega)$ for each $m\in\mathbb{Z}$.

We then define the Hilbert space $K^{-1/2}(\p\Om)$ as a subspace of the measurable function space on $\p\Om$ quotiented via the following equivalence relation:
two functions $\varphi_1,\varphi_2$ on $\partial\Omega$ are considered equivalent and we do not distinguish them if 
$$\frac{1}{2\pi}\int_{\partial\Omega}\varphi_1(\theta)e^{im\theta} d\sigma=\frac{1}{2\pi}\int_{\partial\Omega}\varphi_2(\theta)e^{im\theta} d\sigma\quad \text{for all } m\in\mathbb{Z}.$$
Among all functions in the equivalent class $[\widetilde{\varphi}]$ containing a density function $\widetilde{\varphi}$ defined on $\p\Om$, we take ${\varphi}$ given by
 $${\varphi}=\sum_{m\in\mathbb{Z}}b_m\zeta_m\quad\text{with}\quad b_m=\frac{1}{2\pi}\int_{\partial\Omega}\widetilde{\varphi}(\theta) e^{-im\theta}d\sigma$$
 as the representative of the class $[\widetilde{\varphi}]$.
Now, the Hilbert space $K^{-1/2}(\p\Om)$ is defined as the square integrable sequence space with the basis $\{\zeta_m\}$. In other words, we define
\begin{align*}
K^{-1/2}(\partial\Omega)&:=\left\{\varphi=\sum_{m=-\infty}^\infty b_m \zeta_m(z)\;\Big|\;\  \sum_{m=-\infty}^\infty |b_m|^2<\infty\right\}
\end{align*}
equipped with the inner product
\begin{align}\label{innerproduct_zeta}
&\Big(\sum c_m\zeta_m,\ \sum d_m \zeta_m\Big)_{-1/2}=\sum c_m \overline{d_m}.
\end{align}

For $\p\Om$ of class $C^{1,\alpha}$ with some $\alpha>0$, it holds that
\begin{align*}
K^{-1/2}(\partial\Omega)&=H^{-1/2}(\partial\Omega),
\end{align*}
and the norm $\|\cdot\|_{K^{-1/2}(\partial\Omega)}$ is equivalent
to $\|\cdot\|_{H^{-1/2}(\p\Om)}$; one can find the proof in \cite{jung2018new}. We then define $K^{-1/2}_0(\p\Om)$ as the subspace of $K^{-1/2}(\p\Om)$ with the mean zero condition, i.e., $b_0=0$. It then holds that 
\beq\label{KandH}
K^{-1/2}_0(\partial\Omega)=H^{-1/2}_0(\partial\Omega).
\eeq

\subsection{Series expansion}
The single layer potential for a function $\varphi\in\LtwobdOmega$ is defined by
\begin{equation*}
	\SingleOmega[\varphi](x) = \frac{1}{2\pi}\int_{\partial\Omega}\ln|x-y|\varphi(y)d\sigma(y),\quad x\in\mathbb{R}^2.
\end{equation*}
It satisfies the jump relations on $\p\Om$ with the NP operator $\Kcal^*_{\p\Om}$ defined in \eqnref{eqn:Kstar}: \begin{align*}\label{eqn:jumprelation}
	\SingleOmega[\varphi]\Big|^{+}(x)&=\SingleOmega[\varphi]\Big|^{-}(x)~~~~~~~~\text{a.e. }x\in\partial\Omega,\notag\\[1.5mm]
	\frac{\partial}{\partial\nu}\SingleOmega[\varphi]\Big|^{\pm}(x)&=\left(\pm\frac{1}{2}I+\KstarOmega\right)[\varphi](x)~~~~~~~~\text{a.e. }x\in\partial\Omega.
\end{align*}
The following result is essential to prove the main theorem of the paper. 
\begin{prop}[\cite{jung2018new}]\label{thm:series}
Assume that $\partial\Omega$ is a simply connected bounded domain in $\mathbb{R}^2$ whose boundary $\partial\Omega$ is a piecewise $C^{1,\alpha}$ Jordan curve, possibly with a finite number of corner points without any cusps. Then, the following holds.

\begin{itemize}
\item[\rm(a)]
We have (for $m=0$)
\beq\label{Scal_zeta0}
\SingleOmega[\zeta_0](z)=
\begin{cases}
\ln \gamma \quad &\mbox{if }z\in\overline{\Omega},\\
\ln|w|\quad&\mbox{if }z\in\CC\setminus\overline{\Omega}.
\end{cases}
\eeq 
For $m=1,2,\dots$, we have
	\begin{align}
		\SingleOmega[\zeta_m](z)&=
		\begin{cases}
			\ds-\frac{1}{2\sqrt{m}\gamma^m}F_m(z)\quad&\text{for }z\in\overline{\Omega},\\[2mm]
			\ds-\frac{1}{2\sqrt{m}\gamma^m}\left(\sum_{k=1}^{\infty}c_{m,k}e^{-k(\rho+i\theta)}+\gamma^{2m}e^{m(-\rho+i\theta)}\right)\quad &\text{for } z\in\CC\setminus\overline{\Om},
		\end{cases}\\[3mm]
		\SingleOmega[\zeta_{-m}](z)&=
		\begin{cases}
		\ds	-\frac{1}{2\sqrt{m}\gamma^{m}}\overline{F_{m}(z)}\quad&\text{for }z\in\overline{\Omega},\\[2mm]
		\ds-\frac{1}{2\sqrt{m}\gamma^m}\left(\sum_{k=1}^{\infty}\overline{c_{m,k}}e^{-k(\rho-i\theta)}+\gamma^{2m}e^{m(-\rho-i\theta)}\right)\quad &\text{for } z\in\mathbb{C}\setminus\overline{\Omega}.		
		\end{cases}
	\end{align}
	The series converges uniformly for all $(\rho,\theta)$ such that $\rho\geq\rho_1>\rho_0$.

\item[\rm(b)]
 We have (for $m=0$)
	\beq\label{eqn:Kcal:zeta0}
	\Kcal^*_{\p\Om}[\zeta_0]=\frac{1}{2}\zeta_0.
	\eeq
For $m =1,2,\cdots$
	\begin{align}\label{NP_series1}
		&\KstarOmega[{\zeta_m}](\theta)=\frac{1}{2}\sum_{k=1}^{\infty}\frac{\mu_{k,m}}{\gamma^{m+k}}\, {\zeta}_{-k}(\theta),\quad
		\KstarOmega[{\zeta}_{-m}](\theta)=\frac{1}{2}\sum_{k=1}^{\infty}\frac{\overline{\mu_{k,m}}}{\gamma^{m+k}}\, \zeta_{k}(\theta).
	\end{align}
 The infinite series converges in $K^{-1/2}(\partial\Omega)$. 
	\end{itemize}
\end{prop}

Note that the NP operator $\Kcal^*_{\p\Om}$ is identical to a double infinite matrix via the boundary basis functions $\{\zeta_m\}$, which is self-adjoint thanks to the fact that $\mu_{m,k}=\mu_{k,m}$. Using \eqnref{muineq1_2}, one can show (see \cite{jung2018new}) that 
$$\|\Kcal^*_{\p\Om}\|_{K^{-1/2}(\p\Om)\rightarrow K^{-1/2}(\p\Om)}\leq \frac{1}{2}.$$

\section{Proof of the main theorem}\label{sec:proof}
We investigate the decay behavior for the eigenvalues of the NP operator by using the following two lemmas. The first lemma is an inequality relation between the magnitude of the eigenvalues and the norm of the operator which is perturbed by a finite rank operator; one can find the proof in \cite{little1984eigenvalues}.
\begin{lemma}[Weyl-Courant min-max principle]\label{thm:Weyl-Courant}
	If $T$ is a compact symmetric operator on a Hilbert space, whose eigenvalues $\{\lambda_n\}_{n=1}^\infty$ are arranged as
	$$|\lambda_1| \geq |\lambda_2| \geq \dots \geq |\lambda_n|\geq\dots \rightarrow 0,$$	
	and $S$ is any operator on the Hilbert space of rank $\leq n$, then it holds that
	$$\|T-S\|\geq|\lambda_{n+1}|.$$
\end{lemma}

The second lemma is an estimate of the Grunsky coefficients.
\begin{lemma}{\rm(\cite[Lemma 1.5]{suetin1974polynomials})}\label{lemma:grunskysumbound}
	Let $\Omega$ be a simply connected bounded domain having $C^{1+p,\alpha}$ boundary with $p=0,1,2,\dots$ and $\alpha\in(0,1)$. Then, there exists a positive constant $M$ independent of $s,r$ such that
	\begin{equation}
		\left|\sum_{k=1}^\infty \frac{c_{s,k}\overline{c_{k,r}}}{\gamma^{s+r+2k}}\right|
		\leq
		\frac{M}{s^{p+\alpha} r^{p+\alpha}}\quad\mbox{for all }s,r\geq 1.
	\end{equation}
\end{lemma}

One can express the NP operator in terms of the Grunsky coefficients by Proposition \ref{thm:series}.
By applying the previous two lemmas to this expression, we obtain the following.
\begin{prop}\label{prop:operatornormineq}
	Let $\Omega$ be a simply connected bounded domain having $C^{1+p,\alpha}$ boundary with $p\geq 0$, $\alpha\in(0,1)$. Let $P_N$ be the orthogonal projection from $K^{-1/2}_0(\p\Om)$ onto the $2N$ dimensional subspace spanned by $\{\zeta_{\pm k}\}_{1\leq k\leq N}$. Assuming further $p+\alpha>\frac{1}{2}$, we have
	$$\left\|P_N\KstarOmega-\KstarOmega\right\|_{K^{-1/2}}\leq \frac{C}{N^{p+\alpha-1/2}}\quad\mbox{for all }N\in\NN$$
with a positive constant $C$ independent of $N$.
\end{prop}
\begin{proof}
	Let $\varphi=\sum_{m\neq 0} a_m\zeta_m\in K_0^{-1/2}(\partial\Omega)$ with 
	\beq\label{a:cond}\sum_{m\neq 0}|a_m|^2=1.\eeq
From Proposition \ref{thm:series} and the boundedness of $\Kcal^*_{\p\Om}$ on $K^{-1/2}(\p\Om)$, we have
	\begin{align}\label{eqn:Kaval}
		\KstarOmega[\varphi]=\frac{1}{2}\sum_{k=1}^{\infty}
		\left(\sum_{m=1}^{\infty}\frac{\mu_{m,k}}{\gamma^{m+k}}a_m\right)
		\zeta_{-k}
		+
		\frac{1}{2}\sum_{k=1}^{\infty}
		\left(\sum_{m=1}^{\infty}\frac{\overline{\mu_{m,k}}}{\gamma^{m+k}}a_{-m}\right)
		\zeta_{k}
	\end{align}
and
\begin{align}\label{eqn:PKaval}
		P_N\KstarOmega[\varphi]=\frac{1}{2}\sum_{k=1}^{N}
		\left(\sum_{m=1}^{\infty}\frac{\mu_{m,k}}{\gamma^{m+k}}a_m\right)
		\zeta_{-k}
		+
		\frac{1}{2}\sum_{k=1}^{N}
		\left(\sum_{m=1}^{\infty}\frac{\overline{\mu_{m,k}}}{\gamma^{m+k}}a_{-m}\right)
		\zeta_{k}.
	\end{align}

From \eqnref{a:cond}, \eqnref{eqn:Kaval} and \eqnref{eqn:PKaval}, and by applying the Cauchy-Schwarz inequality, we derive
\begin{align*}
\left\|(\KstarOmega - P_N\KstarOmega)[\varphi]\right\|_{-1/2}^2
&=
\frac{1}{4}\sum_{k=N+1}^\infty 
\left|\sum_{m=1}^{\infty}\frac{\mu_{m,k}}{\gamma^{m+k}}a_m\right|^2
+
\frac{1}{4}\sum_{k=N+1}^\infty
\left|\sum_{m=1}^{\infty}\frac{\overline{\mu_{m,k}}}{\gamma^{m+k}}a_{-m}\right|^2\nonumber\\
&\leq
\frac{1}{2}\sum_{k=N+1}^\infty 
\sum_{m=1}^{\infty}\left|\frac{\mu_{m,k}}{\gamma^{m+k}}\right|^2\\
&=\frac{1}{2}\sum_{k=N+1}^\infty 
\sum_{m=1}^{\infty}
\frac{{c_{k,m}}}{\gamma^{m+k}}\frac{\overline{c_{m,k}}}{\gamma^{m+k}}
.
\end{align*}
From Lemma \ref{lemma:grunskysumbound}, we then obtain
\begin{align*}
\|\KstarOmega - P_N\KstarOmega\|^2
&\leq
 C \sum_{k=N+1}^\infty 
\frac{1}{k^{2(p+\alpha)}}\leq C\frac{1}{N^{2(p+\alpha)-1}}.
\end{align*}
Here, the second term is convergent assuming that $p+\alpha>1/2$. Hence we complete the proof.

\end{proof}

It is worth remarking that 
by Proposition \ref{prop:operatornormineq}, $\KstarOmega$ is a strong limit of a sequence of finite rank operators whenever the boundary $\partial\Omega$ is $C^{1,\alpha}$ for $\alpha>\frac{1}{2}.$

\smallskip

\noindent\textbf{Proof of Theorem \ref{theorem:main}}
As a direct consequence of Proposition \ref{prop:operatornormineq}, equation \eqnref{KandH}, and Theorem \ref{thm:Weyl-Courant}, we prove the theorem.
\qed

\begin{remark}[An ellipse case]
	Consider the conformal mapping of an ellipse
$$\Psi(w)=w+\frac{a}{w},\quad |w|>\gamma$$
for some complex number $a$ and $\gamma$.
 The Faber polynomials associated with the ellipse are
\begin{align*}
	F_0(z)&=1\\
	F_m(z)&=\frac{1}{2^m}\left[\left(z+\sqrt{z^2-4a}\right)^m+\left(z-\sqrt{z^2-4a}\right)^m\right],\quad m=1,2,\cdots.
\end{align*}

For each $m\in\mathbb{N}$, $F_m(\Psi(w))=w^m+\frac{a^m}{w^m}$ so that the Grunsky coefficients are
\begin{equation*}
	 c_{m,k}=\begin{cases}
		a^k&\text{if }k=m,\\
		0&\text{otherwise}.
	\end{cases}
\end{equation*}
From Theorem \ref{thm:series} (c) it follows that
\begin{align}
	&\KstarOmega[{\zeta}_m](z)=\frac{1}{2}\frac{a^m}{\gamma^{2m}}{\zeta}_{-m}(z),\quad \KstarOmega[{\zeta}_{-m}](z)=\frac{1}{2}\frac{\bar{a}^m}{\gamma^{2m}}{\zeta}_{m}(z).
\end{align}
Hence, $\KstarOmega$ has the eigenvalues and the corresponding eigenfunctions
 $$\pm \frac{1}{2}\frac{|a|^m}{\gamma^{2m}},\quad \pm\left(\frac{{a}}{|a|}\right)^m{\zeta}_{-m}+{\zeta}_{m},\quad m=1,2,\dots.$$

\end{remark}


\end{document}